\documentclass[11pt]{amsart}

\usepackage{latexsym,amssymb,amsmath,graphics,amsfonts,amsthm,amsopn,cite,mathrsfs}
\usepackage{verbatim}

\usepackage[dvips]{graphicx}
\usepackage{xcolor}

\def\ff{{\mathcal F}}

\def\lll{{\mathcal L}}

\def\ffi{\varphi}
\def\eps{\varepsilon}
\def\dst{\displaystyle}

\renewcommand{\Im}{\mathrm{Im}\,}

\DeclareMathOperator{\Real}{Re}

\def\N{{\mathbb{N}}}

\def\R{{\mathbb{R}}}
\def\S{{\mathbb{S}}}

\def\Z{{\mathbb{Z}}}

\def\fg{{\mathfrak{g}}}
\def\fz{{\mathfrak{z}}}

\def\d{\,{\mathrm{d}}}
\newcommand{\abs}[1]{{\left|{#1}\right|}}
\newcommand{\scal}[1]{{\left\langle{#1}\right\rangle}}

\newtheorem{lemma}{Lemma}[section]
\newtheorem{proposition}[lemma]{Proposition}
\newtheorem{theorem}[lemma]{Theorem}
\newtheorem{corollary}[lemma]{Corollary}

\theoremstyle{definition}
\newtheorem{definition}[lemma]{Definition}

\newtheorem{hypothesis}{Hypothesis}

\theoremstyle{remark}
\newtheorem{remark}[lemma]{Remark}

\begin{document}

\title[Uncertainty principle on $H$-type groups]{An uncertainty principle for solutions of the Schr\"odinger equation on $H$-type groups}


\author{Aingeru Fern\'{a}ndez-Bertolin, Philippe Jaming \& Salvador P\'erez-Esteva}

\address{A.F.-B.\,: Departamento de Matem\'aticas, Universidad del Pa\'is Vasco UPV/EHU, Apartado 644, 48080, Bilbao, Spain} 

\email{aingeru.fernandez@ehu.eus}

\address{P.J. \,: Univ. Bordeaux, IMB, UMR 5251, F-33400 Talence, France.
CNRS, IMB, UMR 5251, F-33400 Talence, France.}
\email{Philippe.Jaming@math.u-bordeaux.fr}

\address{S.P.E.\,: Instituto de Matem\'{a}ticas, Unidad Cuernavaca\\Universidad Nacional Aut\'{o}noma de M\'{e}xico\\
Cuernavaca\\ Morelos 62251\\ M\'{e}xico}
\email{salvador@matcuer.unam.mx}

\begin{abstract}
In this paper we consider uncertainty principles for solutions of certain PDEs on $H$-type groups. We first prove that, contrary to the euclidean setting, the heat kernel on $H$-type groups is not characterized as the only solution of the heat equation that has sharp decay at 2 different times.

We then prove the analogue of Hardy's Uncertainty Principle for solutions of the Schr\"odinger equation with potential on $H$-type groups. This extends the free case considered by Ben Sa\"id, Dogga and Thangavelu \cite{BTD} and by Ludwig and M\"uller \cite{LM}.
\end{abstract}



\maketitle


\section{Introduction}

The aim of this paper is to prove a version of Hardy's Uncertainty Principles (UP) on $H$-type groups. Let us recall that Hardy's uncertainty principle states that a function $f$ and its Fourier transform
$$
\hat{f}(\xi)=\frac{1}{\sqrt{2\pi}}\int_\R e^{-i\xi\cdot x }f(x)\d x, \xi\in \R,
$$
cannot both have fast decay simultaneously:
\begin{theorem}[Hardy \cite{Ha}]\label{thm:Hardy}
Let $f$ be a function such that $f(x)=O(e^{-x^2/\beta^2})$ and $ \hat{f}(\xi)=O(e^{-4\xi^2/\alpha^2})$.
If $1/\alpha\beta>1/4$ then $f\equiv0$ while if $1/\alpha\beta=1/4$, $f$ is a constant multiple of $e^{-x^2/\beta^2}$.
\end{theorem}

There is a vast literature concerning extensions of this theorem to other settings. Among the references related to this paper, let us mention the extensive work by Baklouti, Kaniuth, Thangavelu and coauthors, \cite{BK,BT,KK,PT,SST,Th1,Th2}, where extensions to various Lie groups have been proven (see also the survey \cite{FS} and the book \cite{Th3}). One of the aims of those papers is to replace the Gaussian in Theorem \ref{thm:Hardy} by the heat kernel. One difficulty when going to a Lie group $G$ is that, contrary to the case $G=\R^n$, the Fourier transform is not a function defined on $G$ itself. In particular, there is no direct substitute for the decrease condition on the Fourier transform  in Hardy's theorem. 

One way to overcome this difficulty is to restate Hardy's uncertainty principle in terms of solutions of a PDE. Let us first mention how Hardy's theorem can be reformulated in terms of solutions of the free heat equation :

\medskip

{\em Let $u$ be a solution of $\partial_t u=\Delta u$. If $u(x,0)$ and $e^{|x|^2/\delta^2}u(x,1)$ are in $L^2(\R^n)$ for some $\delta\le 2$, then $f\equiv0$. Further, if $u(x,0)\in L^1(\R^n)$ and $e^{|x|^2/4}u(x,1)\in L^\infty(\R^n)$, then $u(x,0)$ is a constant multiple of the heat kernel, $e^{-|x|^2/2}$.}

\medskip

Indeed, this follows directly from Theorem \ref{thm:Hardy} applied to $u(x,1)$, see \cite{Ja} and the introduction of \cite{EKPV}. This result gives a characterization of the heat kernel in terms of the optimal decrease of a solution of the heat equation. One of our goals in this paper is to show that this characterization is no longer true on $H$-type groups, by proving the following
\begin{theorem}[Limiting case of Hardy's Uncertainty Principle]
\label{th:intro1}
For every $s_0$, there exists a non-zero, non-negative function $f$, that is not a multiple of $p_{s_0}$ such that, 
for every $s\geq s_0$, the solution $u$ of $\dst\begin{cases}\partial_su=\lll u\\ u(x,0)=f\end{cases}$
is such that $u(x,s)\leq p_{s_0+s}$.
\end{theorem}

This shows that the heat kernel can not be characterized via a restatement of Hardy's Uncertainty Principle in terms of solutions of the heat equation. An other possible direction is to restate Hardy''s Uncertainty Principle in terms of solutions of the free Schr\"odinger equation. This direction has first been investigated by Chanillo \cite{Ch}. He first notices that Hardy's UP can be rested as follows:

\medskip

{\em Let $u$ be a solution of $i\partial_t u+\Delta u=0$. 
If $u(x,0)=O(e^{-|x|^2/\beta^2}),\ u(x,1)=O(e^{-|x|^2/\alpha^2})$ and $\alpha\beta<4$, then $u\equiv0$. Also, if $\alpha\beta=4$, $u$ has as initial datum a constant multiple of $e^{-(1/\beta^2+i/4)|x|^2}$.}

\medskip

Chanillo then skillfully transfers this result to solutions of the Schr\"odinger equation on complex solvable Lie groups (up to the end point $\alpha\beta=4$). This strategy of proof has also been adopted by Ben Sa\"id, Dogga and Thangavelu in \cite{BTD} to extend Chanillo's version of Hardy's UP to $H$-type groups and, independently and almost simultaneously, by 
Ludwig, M\"uller \cite{LM} for general step 2 nilpotent Lie groups.
Our aim here is to further extend the result in \cite{BTD} to cover solutions of linear Schr\"odinger equations of the form
\begin{equation}
\label{eq:SchrodHtypeintro}
i\partial_t u(g,t)+\mathcal{L}_G u(g,t)+V(g,t)u(g,t)=0,\ (g,t)\in G\times [0,T]
\end{equation}
where $G$ is an $H$-type group and with some restrictions on the potential $V$.

In the Euclidean case, this extension was given by Escauriaza, Kenig, Ponce and Vega \cite{EKPV,EKPV2} where the authors developed a machinery based on real variable calculus in order to prove Hardy's uncertainty principle, which typically is proved via complex analysis. 
This method has then been adapted to different settings, including to the Magnetic Schr\"odinger equation \cite{bfgrv,CF1,CF}. By using a Radon transform we have been able to reduce the setting of
of the Schr\"odinger equation on $H$-type groups to the one in \cite{CF}. Doing so
our main result concerning solutions of Schr\"odinger equations is then the following:

\begin{theorem}
\label{th:intro2}
Let $G$ be a $H$-type group.
Let $T>0$ and $s,\sigma>0$. Let $V$ be a bounded real-valued potential that is independent on time and on the central variable.
Let $u\in C^1([0,T], H^1(G))$ be a solution of \eqref{eq:SchrodHtypeintro}. Assume that there is a $C>0$ such that
\begin{eqnarray*}
|u(x,z,0)|&\le& C p_s(x,z),\\
|u(x,z,T)|&\le& C p_{\sigma}(x,z).
\end{eqnarray*}
If $s\sigma<T^2$, then $u\equiv0$.
\end{theorem}

The actual result is more general as it allows some complex and time-dependent potenitals, provided they go sufficiently fast to 0 at infinity.
For the precise statement that requires the introduction of several notations, see Theorem \ref{th:SchroHtypeUP}.  In the case $V=0$ we recover the result in \cite{BTD}.

The remaining of the paper is organized as follows: In Section 2 we introduce $H$-type groups as well as the different concepts we need to explain in order to state and prove our results. This section concludes with the proof of the limiting case of Hardy's Uncertainty Principle. In Section 3 we turn to the study of the linear Schr\"odinger equation using real variable methods. By using different reductions  first to the Heisenberg group and then to the Euclidean setting, we give sufficient decrease conditions on a solution to vanish identically. Then, we prove our main result in this section by relating these decrease conditions to the decrease of the heat kernel.

\section{$H$-type groups}

In this section, we gather the necessary information we will need on $H$-type groups.
$H$-type groups were first introduced in \cite{Ka}. \cite[Chapter 18]{BLU} contains an extended development of their 
fundamental properties; some of them being further extended in \cite{El}. We follow closely the presentation of this last paper here, and refer the reader to \cite{BLU,El} for further details.
Note that, for briefness of presentation, we present as definition some results that actually need proofs.

For elements of $\R^k$ we denote by $|\cdot|$ the Euclidean norm and by $\scal{\cdot,\cdot}$
the Euclidean scalar product. We write $\N=\{0,1,2,\ldots\}$ for the non-negative integers.
For $\alpha=(\alpha_1,\ldots,\alpha_k)\in\N^k$
we use the classical multi-index notations $|\alpha|=\alpha_1+\cdots+\alpha_k$.

If $f,g$ are two functions $X\to\R$, we write $f (x)\lesssim g(x)$ ---resp. $f (x)\simeq g(x)$--- to mean there exist finite positive constants $C_1$, $C_2$ such that $f (x)\leq C_2g(x)$
---resp. 
$C_1g(x)\leq f (x)\leq C_2g(x)$--- for all $x\in X$.

\subsection{Generalities}

\begin{definition}
Let $\fg$ be a finite-dimensional real Lie algebra with center $\fz\not=0$. We say $\fg$ is of 
\emph{$H$-type} (or Heisenberg type) if $\fg$ is equipped with an inner product $\scal{\cdot,\cdot}$
such that:
\begin{enumerate}
\item $[\fz^\perp,\fz^\perp]=0$ and;
 
\item for each $z\in\fz$, define $J_z\,:\fz^\perp\to\fz^\perp$ by
$$
\scal{J_zx,y}=\scal{z,[x,y]}
$$
whenever $x,y\in \fz^\perp$. Then $J_z$ is orthogonal whenever $\scal{z,z}=1$.
\end{enumerate}
An \emph{$H$-type group} is a connected, simply connected Lie group whose Lie algebra is of $H$-type.
\end{definition}

It can be shown that an $H$-type group is a stratified 2-step nilpotent Lie group.
In particular, we have an orthogonal decomposition $\fg=\fz^\perp\oplus\fz$ such that
$[\fz^\perp,\fz^\perp]=\fz$. Moreover, $\fz^\perp$ has even dimension.
We will denote by $2n=\dim\fz^\perp$ and $m=\fz$ and we identify
$\fz^\perp\simeq\R^{2n}$ and $\fz=\R^m$. It turns out that $m$ and $n$ can not be arbitrary. Indeed,
writing $2n=a2^{4p+q}$ where $a$ is odd and $0\leq q\leq 3$, then $\R^{2n}\times\R^m$ can be endowed with a Lie algebra structure
of $H$-type as above if and only if $m<\rho(2n):=8p+2^q$.

As $G$ is step 2 nilpotent, $G$ can be identified with $\fg=\R^{2n+m}$ as set.
An element $g\in G$ is thus of the form $g=(x,z)$ with $x\in\R^{2n}$,
$z\in\R^m$. According to the Baker--Campbell--Hausdorff formula, the group operation is then given by
$$
(x,z)\cdot(y,z')=(x+y,z+z'+\frac{1}{2}[x,y]).
$$
The identity of $G$ is $(0,0)$, and the inverse operation is given by $(x,z)^{-1}=(-x,-z)$.
The maps $\{J_z:\ z\in\R^m\}$ are identified with $2n\times 2n$ skew-symmetric matrices which are orthogonal when $|z| = 1$.
In particular, $|J_zx|=|x|\,|z|$.

Next, for $a\in\R\setminus\{0\}$, we define the dilations $\ffi_a(x,z)=(a x,a^2 z)$. Note that this is 
both a group and Lie algebra automorphism. We also denote by $\ffi_a$ the action of $\ffi_a$ on 
functions: for a function $f$ on $\R^{2n+m}$ we denote $\ffi_a\cdot f(x,z)=f(a x,a^2 z)$.

We let $\{e_1,\ldots,e_{2n}\}$ denote the standard basis for $\R^{2n}$, and $\{u_1,\ldots, u_m\}$ denote the standard basis for $\R^m$.

The Haar measure on $G$ is simply the Lebesgue measure on $\R^{2n+m}$ and convolution of two functions is given by
$$
f*g(x,z)=\int_G f(\omega)g\bigl(\omega^{-1}(x,z)\bigr)\d\omega
=\int_{\R^{2n}}\int_{\R^m}f(u,v)g(x-u,z-v+\frac{1}{2}[x,u])\d u\d v.
$$

We can now identify $\fg$ with the set of left-invariant vector fields on $G$, where $X_i(0) =
\dst\frac{\partial}{\partial x_i}$ and $Z_i(0)=\dst\frac{\partial}{\partial z_i}$;
then $\mbox{span}\,\{X_1, . . . , X_{2n}\} =\fz^\perp$,
$\mbox{span}\,\{Z_1, . . . , Z_m\} =\fz$. A computation shows that
$$
X_i=\frac{\partial}{\partial x_i}+\frac{1}{2}\sum_{j=1}^m\scal{J_{u_j}x,e_i}\frac{\partial}{\partial z_j}
\quad\mbox{and}\quad
Z_j=\frac{\partial}{\partial z_j}.
$$
The elementary Lie brackets are given by
$$
[X_i,X_j]=\sum_{k=1}^m\scal{J_{u_k}e_i,e_j}Z_k
$$
and all other elementary brackets are $0$. 

Note that the $X_i$'s are homogeneous of degree 1 and the $Z_j$ are homogeneous of degree 2 with respect to 
$\ffi_a$: $X_i(\ffi_a\cdot f)=a \ffi_a\cdot(X_if)$, $Z_j(\ffi_a\cdot f)=a^2 \ffi_a\cdot(Z_jf)$. We will use the 
following notation: if $\alpha=(\alpha_1,\ldots,\alpha_{2n},\alpha_{2n+1},\ldots,\alpha_{2n+m})\in\N^{2n+m}$ then
$\mathcal{X}^\alpha=X_1^{\alpha_1}\cdots X_{2n}^{\alpha_{2n}}Z_1^{\alpha_{2n+1}}\cdots Z_m^{\alpha_{2n+m}}$
and $w(\alpha)=\alpha_1+\cdots+\alpha_{2n}+2\alpha_{2n+1}+\cdots+2\alpha_{2n+m}$.


\begin{definition}
The \emph{sublaplacian} $\lll$ for $G$ is the operator given by:
\begin{equation}
\label{eq:sublaplacian}
\lll=\sum_{j=1}^{2n}X_j^2.
\end{equation}
\end{definition}


\subsection{Carnot-Carath\'eodory distance}

Recall (see, e.g., \cite{Var}) that the Carnot-Carath\'eodory distance to the origin
associated to the sum of squares operator $\lll$ is defined by
$$
d(x,z):=\inf_\gamma|\gamma|
$$
where the infimum is taken over all absolutely continuous curves $\gamma\,:[0,1]\to G$
which are horizontal and connect $0$ with $(x,z)$ {\it i.e.} $\gamma(0)=0$, $\gamma(1)=(x,z)$
and which are horizontal, that is
$$
\gamma'(t)=\sum_{j=1}^{2n}
a_j(t)X_j\bigl(\gamma(t)\bigr)
$$
for a.e. $t \in [0, 1]$. Here, $|\gamma|$ denotes the length of $\gamma$ given by
$$
|\gamma|=\int_0^1\left(\sum_{j=1}^{2n}|a_j(t)|^2\right)^{\frac{1}{2}}\,\mbox{d}t
$$
We then define the \emph{Carnot-Carath\'eodory distance $d$} via the formula
$d(g,h)=d(g^{-1}h)$ where we use the same letter $d$ to designate the distance and
the distance to $0$. Note that, by definition,
$d$ is left-invariant, {\it i.e.} $d(g,h) = d(kg, kh)$ for every $g,k,h\in G$

The Carnot-Carath\'eodory distance $d$ can be computed explicitely on $G$ as follows:
Define the function $\nu\,:[0,\pi[\to\R$ by $\nu(0)=0$ and, for $\theta\in]0,\pi[$,
$$
\nu(\theta)=-\frac{\d}{\d\theta}[\theta\cot\theta]=\frac{\theta-\sin\theta\cos\theta}{\sin^2\theta}=\frac{2\theta-\sin2\theta}{1-\cos 2\theta}.
$$
Then \cite[Theorem 3.5]{El}
$$
d(x,z)=\begin{cases}
|x|\frac{\theta}{\sin\theta}&\mbox{if }x\not=0,z\not=0\\
|x|&\mbox{if }z=0\\
\sqrt{4\pi|z|}&\mbox{if }x=0
\end{cases}
$$
where $\theta$ is the unique solution in $[0,\pi[$ of the equation $\nu(\theta)=\dst\frac{4|z|}{|x|^2}$.

Note that $d\bigl(\ffi_a(x,z)\bigr)=a d(x,z)$ from which it is not difficult to show
that $d(x, z)\simeq |x| + |z|^{1/2}$. Equivalently, $d(x, z)^2\simeq |x|^2 + |z|$.
One can get a more precise result:

\begin{lemma}
For every $(x,z)\in G$, 
\begin{equation}
\label{eq:equivdist}
\frac{\pi}{4}(|x|^2+4|z|)\leq d(x,z)^2\leq \pi(|x|^2+4|z|).
\end{equation}
Moreover, the constants in this inequality are optimal.

Further, for every $\eps\in(0,1)$, there exists a constant $C_\eps$ such that, for every $(x,z)\in G$,
\begin{equation}
\label{eq:equivdistprecise}
(1-\eps)|x|^2+\pi\eps|z|\leq d(x,z)\leq (1+\eps)|x|^2+\frac{20\pi}{\eps}|z|.
\end{equation}
\end{lemma}

The first part of this lemma {\it i.e.} \eqref{eq:equivdist}, is well known \cite{Var}. We here take the opportunity to find the best constants. The estimate \eqref{eq:equivdistprecise} can be found in \cite{LM} in a less precise form.

\begin{proof} The inequalities are obvious when $x=0$ or $z=0$ so we assume that $x,z\not=0$.

Let $F\,:[0,\pi]\to\R$ be defined by $F(0)=1$ and, for $\theta\in]0,\pi]$,
$$
F(\theta)=\frac{\theta^2}{\sin^2\theta+\theta-\sin\theta\cos\theta}=\frac{\theta^2}{\theta+\sin\theta(\sin\theta-\cos\theta)}.
$$

Note that, if $\theta$ is the unique solution in $[0,\pi[$ of $\nu(\theta)=\dst\frac{4|z|}{|x|^2}$, then
$$
F(\theta):=\frac{(\theta/\sin\theta)^2}{1+\nu(\theta)}=\frac{d(x,z)^2}{|x|^2+4|z|}.
$$
It is therefore enough to check that $\pi/4=F(\pi/4)\leq F(\theta)\leq F(\pi)=\pi$.

But
\begin{eqnarray*}
F'(\theta)&=&\frac{2\theta}{\theta+\sin\theta(\sin\theta-\cos\theta)}\\
&&-\frac{\theta^2}{\bigl(\theta+\sin\theta(\sin\theta-\cos\theta)\bigr)^2}\bigl(1+\cos\theta(\sin\theta-\cos\theta)+\sin\theta(\sin\theta+\cos\theta)\bigr)\\
&=&\frac{2\theta}{\bigl(\theta+\sin\theta(\sin\theta-\cos\theta)\bigr))^2}\times\\
&&\quad\times\bigl[2\bigl(\theta+\sin\theta(\sin\theta-\cos\theta)\bigr)-\theta\bigl(1+\cos\theta(\sin\theta-\cos\theta)+\sin\theta(\sin\theta+\cos\theta)\bigr)
\bigr]
\end{eqnarray*}
which has same sign as
$$
\psi(\theta):=\theta\bigl(1-\sin\theta(\sin\theta+\cos\theta)\bigr)+(2\sin\theta-\theta\cos\theta)(\sin\theta-\cos\theta).
$$
Notice that $\psi(\pi/4)=0$ and that
\begin{eqnarray*}
\psi'(\theta)&=&1+\theta
+(\cos\theta+\theta\sin\theta)(\sin\theta-\cos\theta)
+(\sin\theta-\theta\cos\theta)(\sin\theta+\cos\theta)\\
&=&\theta(1-\cos 2\theta-\sin 2\theta)+1-\cos 2\theta+\sin 2\theta.
\end{eqnarray*}
Now $\psi'(\pi/4)=0$ and
$$
\psi''(\theta)=1+\cos 2\theta+\sin 2\theta+2\theta(\sin 2\theta-\cos 2\theta)
=1+\cos(2\theta)(1-2\theta)+\sin 2\theta(1+2\theta).
$$
From this, it is obvious that $\psi''\geq 2-\pi/2$ on $[0,\pi/4]$, thus $\psi'\leq 0$ on $[0,\pi/4]$
therefore $\psi$ thus $F$ decreases on $[0,\pi/4]$. In particular, $\pi/4=F(\pi/4)\leq F(\theta)\leq F(0)=1$
for $\theta\in[0,\pi/4]$.

For $\theta\in[\pi/4,3\pi/4]$ one easily checks that $1-\cos 2\theta-\sin 2\theta\geq 0$
and $1-\cos 2\theta+\sin 2\theta\geq 0$ thus $\psi'\geq 0$, thus $\psi$ increases and is therefore non-negative.
This in turn means that $F'$ is non negative, thus $\pi/4=F(\pi/4)\leq F(\theta)\leq F(3\pi/4)$ for
$\theta\in[\pi/4,3\pi/4]$. 

Finally, on $[3\pi/4,\pi]$, $1+\cos(2\theta)+\sin2\theta\leq 1$ while $\cos(2\theta)-\sin(2\theta)\leq -1$
thus $\psi''(\theta)\leq 2(1-\theta)\leq 0$. Therefore $\psi'(\theta)\geq\psi'(\pi)=0$ and $\psi$ is increasing,
$\psi(\theta)\geq\psi(3\pi/4)=3\pi/4+2>0$. It follows that $F$ is increasing on $[3\pi/4,\pi]$ thus
$F(\pi/4)\leq F(3\pi/4)\leq F(\theta)\leq F(\pi)=\pi$ for $\theta\in [3\pi/4,\pi]$.

The proof of \eqref{eq:equivdistprecise} then follows the steps of \cite{LM}: on one hand,
$$
d(x,z)^2\geq \max\bigl(d(x,0)^2,d(0,z)^2\bigr)\geq \max(|x|^2,\pi|z|)\geq (1-\eps)|x|^2+\pi\eps|z|.
$$
On the other hand
$$
d(x,z)^2\leq \bigl(d(x,0)+d(0,z)\bigr)^2\leq(|x|+\sqrt{4\pi|z|})^2\leq (1+\eps)|x|^2+\frac{20\pi}{\eps}|z|.
$$
as claimed.
\end{proof}

\subsection{Heat kernel}

\begin{definition}
The heat kernel $p_t$ for $G$ is the unique fundamental solution to the corresponding heat
equation $\dst\left(\lll-\frac{\partial}{\partial t}\right)u=0$
that is, $p_t = e^{t\lll}\delta_0$, where $\delta_0$ is the Dirac delta distribution supported at $0$.

In particular, if $\dst\left(\lll-\frac{\partial}{\partial t}\right)u=0$ and $u(x,z,0)=u_0(x,z)\in L^2(\R^{2n+m})$
then $u(x,z,t)=u_0*p_t(x,z)$.
\end{definition}

Our next step is to record an explicit formula for $p_t (x, z)$. Various derivations of this formula appear in the literature.
For general step 2 nilpotent groups, \cite{Ga} derived such a formula probabilistically from a formula in \cite{Le} regarding
the L\'evy area process. Another common approach, worked out in \cite{DP}, involves expressing $p_t$ as the Fourier transform
of the Mehler kernel. Taylor \cite{Ta} has a similar computation. Other approaches have involved complex Hamiltonian
mechanics \cite{BGG}, magnetic field heat kernels \cite{Ki}, and approximation of Brownian motion by random walks 
\cite{Hu}. Of particular interest to us is the approach by Randall \cite{Ra} who obtains the formula for H-type groups as the
Radon transform of the heat kernel for the Heisenberg group.
In our notation, we find that
\begin{equation}
\label{eq:hul}
p_t(x,z)=t^{-n-m}\int_{\R^m}
e^{-\frac{\pi}{2}|\lambda|\coth(2\pi |\lambda|)|t^{-1/2}x|^2}
\left(\frac{|\lambda|/2\pi}{2\sinh(2\pi |\lambda|)}\right)^n
e^{2i\pi\scal{\lambda,t^{-1}z}}\d \lambda\\
\end{equation}
Also this is not obvious from the above formula, $p_t$ is non-negative and $p_t*p_{t'}=p_{t+t'}$.
Note that $p_t(x,z)=t^{-m-n}p_1(t^{-1/2}x,t^{-1}z)$ {\it i.e.}  $p_t=t^{-m-n}\ffi_{t^{-1/2}}p_1$.

Precise estimates of the heat kernel have been obtained by Elridge \cite{El}:
\begin{equation}
\label{eq:elredge}
p_t(x,z)\simeq 
t^{-n-m}\frac{\bigl(1+t^{1/2}d(x,z)\bigr)^{2n+m-1}}{1+\big(t|x|d(x,z)\big)^{n-1/2}}e^{-\frac{1}{4t}d(x,z)^2}.
\end{equation}
Further, Randall \cite{Ra} has shown that  $p_t$ admits an analytic extension with respect to the time parameter $t$ to the half-complex plane $\Real(t)>0$.
We will need a slightly less precise estimate than \eqref{eq:elredge}, but that will also be valid for complex time $t$:

\begin{lemma}[Estimate of the heat kernel]
\label{lem:estheat}
For every $\eps\in(0,1)$, there exists $C=C(\eps)$ such that, for every $t>0$ 
$$
\frac{1}{C t^{n+m}}\exp\left(-\frac{1}{4t}\bigl((1+\eps)|x|^2+\frac{20\pi}{\eps}|z|\bigr)\right)
\leq p_t(x,z)\leq \frac{C}{ t^{n+m}}\exp\left(-\frac{1}{4t}\bigl((1-\eps)|x|^2+\pi\eps|z|\bigr)\right).
$$
Moreover, if $t\in\mathbb{C}$ with $\Real(t)>0$, 
$$
p_t(x,z)\leq \frac{C}{ \Real(t)^{n+m}}\exp\left(-\Real\frac{1}{4t}\bigl((1-\eps)|x|^2+\pi\eps|z|\bigr)\right).
$$

\end{lemma}

\begin{proof}
For $t$ real, this formula follows directly from incorporating \eqref{eq:equivdistprecise} into \eqref{eq:elredge} and is well known up to the numerical constants (see \cite[p 50]{Var}).

For complex $t$, we use the fact that $p_t$ admits an analytic extension and the result follows by an application of Phragm\'en-Lindel\"of as in Theorem 3.4.8 of \cite{Da}.
\end{proof}

We are now in position to prove our first result:

\begin{theorem}
\label{th:gaussian}
For every $x\in\R^n$, there exists a non-negative finite measure $\nu_x$ on $[0,+\infty)$ such that
$$
\int_0^{+\infty}\frac{1}{\tau}\,\mbox{d}\nu_x(\tau)<+\infty
$$
and for every $z\in\R^m$,
$$
p_1(x,z)=\int_0^{+\infty}\frac{1}{\tau}e^{-\pi|z|^2/\tau^2}\d\nu_x(\tau).
$$
\end{theorem}

\begin{proof}
Let us write
$$
P_{n,m}(x,z)=\int_{\R^m}
e^{-\frac{\pi}{2}|\lambda|\coth(2\pi |\lambda|)|x|^2}
\left(\frac{|\lambda|/2\pi}{2\sinh(2\pi |\lambda|)}\right)^n
e^{2i\pi\scal{\lambda,z}}\d \lambda.
$$
In particular, $P_{n,1}$ is the heat kernel of the Heisenberg group, as shown by Hulanicki \cite{Hu}. Randall \cite{Ra} showed that
\begin{equation}
\label{eq:radon}
P_{n,m}(x,z)=c_{n,m}\int_{\S^{m-1}}P_{n,1}(\scal{z,\theta})\d\sigma(\theta)
\end{equation}
where $\sigma$ is the normalized Lebesgue measure on the unit sphere $\S^{m-1}$ of $\R^m$ and $c_{n,m}$ is a normalization constant. Further, when $m<\rho(2n)$, Randall has shown that $P_{n,m}$ is the heat kernel on $\R^{2n}\times\R^m$ with the $H$-type group structure
mentioned above.

One consequence of \eqref{eq:radon} is that, for every $m$, $P_{n,m}$ is a non-negative function.
On the other hand, let us denote by $\ffi_x$ the function defined on $\R$ by
$$
\ffi_x(u)=e^{-\frac{\pi}{2}u\coth(2\pi u)|x|^2}
\left(\frac{u/2\pi}{2\sinh(2\pi u)}\right)^n.
$$
Then $P_{n,m}$ is the Fourier transform of the radial extension to $\R^m$ of $\ffi_x$, {\it i.e.}
$P_{n,m}(x,z)=\ff[\ffi_x(|\lambda|)](z)$ where $\ff$ is the usual Fourier transform on $\R^m$.
It then follows from a celebrated theorem of Schoenberg \cite{Sc} (see also \cite{SvP} for a more
modern proof and further references), that $\ffi_x$ is an average of Gaussians with respect to some
finite non-negative measure $\nu_x$
$$
\ffi_x(u)=\int_0^{+\infty}e^{-\pi \tau^2u^2}\,\mbox{d}\nu_x(\tau).
$$
One easily checks that $\int_0^\infty\ffi_x(u)\,\mbox{d}u$ is finite, thus Fubini's theorem yelds
$$
\int_0^\infty\ffi_x(u)\,\mbox{d}u=\int_0^{+\infty}\int_0^\infty e^{-\pi \tau^2u^2}\,\mbox{d}u\,\mbox{d}\nu_x(\tau)
=\int_0^{+\infty}\frac{1}{\tau}\,\mbox{d}\nu_x(\tau)<+\infty.
$$
It follows from Fubini's theorem again that
$$
p_1(x,z)=\int_{\R^m}\int_0^{+\infty}e^{-\pi\tau^2|\lambda|^2}\d\nu_x(\tau)e^{2i\pi\scal{\lambda,z}}\d \lambda
=\int_0^{+\infty}\frac{1}{\tau}e^{-\pi|z|^2/\tau^2}\d\nu_x(\tau)
$$
as claimed.
\end{proof}

\begin{remark}
It would be nice if the measure $\nu_x$ could be determined explicitly. Our attempts to do so have not succeeded.
\end{remark}

We are now in position to prove that the heat kernel on $H$-type groups is not characterized 
via an optimal decrease condition at two different times. This is in strong contrast with the Euclidean case as shows Theorem \ref{th:intro1} that we recall for the reader's convenience:

\begin{theorem}[Limiting case of Hardy's Uncertainty Principle]
For every $s_0$, there exists a non-zero, non-negative function $f$, that is not a multiple of $p_{s_0}$ such that, 
for every $s\geq s_0$, the solution $u$ of $\dst\begin{cases}\partial_su=\lll u\\ u(x,0)=f\end{cases}$
is such that $u(x,s)\leq p_{s_0+s}$.
\end{theorem}

\begin{proof} We will only prove it for $s_0=1$, a straightforward scaling argument gives the general result.

Recall from Theorem \ref{th:gaussian} that 
$$
p_1(x,z)=\int_0^{+\infty}\frac{1}{\tau}e^{-\pi|z|^2/\tau^2}\d\nu_x(\tau)
$$
with $\nu_x$ a finite measure such that
$$
\int_0^{+\infty}\frac{1}{\tau}\,\mbox{d}\nu_x(\tau)<+\infty.
$$
In particular, $\nu_x$ has essentially no mass at $0$. However, if there were an $a_x$ such that $\nu_x([0,a_x])=0$ then we would have
$\ffi_x\lesssim e^{-\pi a_x^2|u|^2}$ which is obviously not the case.
On the other hand, if the support of $\nu_x$ where included in $[0,a_x]$, then $\ffi_x$ would extend into an entire 
function of order $2$, $|\ffi_x(u+iv)|\lesssim e^{\pi a_x^2 v^2}$. But this is again not the case
since 
$$
|\ffi_x(iu)|=\left(\frac{|u|/2\pi}{|\sin 2\pi u|}\right)^n\to +\infty
$$
when $u\to k/2$, $k\in\Z\setminus\{0\}$.

Let us now consider $f$ defined by
$$
f(x,z)=\int_1^{+\infty}\frac{1}{\tau}e^{-\pi|z|^2/\tau^2}\d\nu_x(\tau).
$$

From the properties of the measure $\nu_x$, we get that $f\not=0$, $0\leq f\leq p_1$ and, according to Elridge's 
estimates of $p_t$ and the fact that $f(x,z)\leq e^{-\pi|z|^2}$, we get $f(x,z)=o\bigl(p_1(x,z)\bigr)$ when $|z|\to\infty$,
so that $f$ is not a multiple of $p_1$. Further $f*p_t\lesssim p_1*p_t=p_{t+1}$. 
\end{proof}

\section{The Schr\"odinger equation in $H$-type groups: A "real" approach}

In this section we deal with a solution of the Schr\"odinger equation for the sub-Laplacian of an $H$-type group $G$
(isomorphic to $\mathbb{R}^{2n}\times\mathbb{R}^m$),
\begin{equation}\label{eq:SchrodHtype}
i\partial_tu(g,t)+\mathcal{L}_Gu(g,t)+V(g,t)u(g,t)=0,\ \ (g,t)\in G\times[0,T],
\end{equation}
where $g=(x,z),\ x\in\mathbb{R}^{2n},\ z\in\mathbb{R}^m$. We assume that the potential satisfies the following hypothesis.

\begin{hypothesis}\label{potential}
The potential $V$ is independent of the central variable and can be written in the form
$V(x,z,t)=V_1(x)+V_2(x,t)$  where, 
\begin{itemize}
\item $V_1$ is a real-valued bounded potential;

\item for some $a,b,T>0$,
\begin{equation}
\label{eq:CFcond}
\sup_{t\in[0,T]}\sup_{x\in\mathbb{R}^{2n}}|e^{\frac{T^2|x|^2}{(a t+b(T-t))^2}}V_2(x,t)| 
<+\infty.
\end{equation}
\end{itemize}
\end{hypothesis}

Note that, if $V$ satisfies Hypothesis \ref{potential} for some $a,b$ then it also satisfies the hypothesis for $a'\geq a$ and $b'\geq b$.
Also \eqref{eq:CFcond} implies that
$$
\sup_{t\in[0,T]}\sup_{x\in\mathbb{R}^{2n}}|\Im V_2(x,t)|<+\infty.
$$

\begin{theorem}
\label{th:SchroHtypeUP}
Let $a,b,T>0$ and let $V$ be a potential satisfying Hypothesis \ref{potential} with parameters $a,b,T$.
Let $u\in C^1([0,T], H^1(G))$ be a solution of \eqref{eq:SchrodHtype}. Assume that there are $c,C>0$ and a function $g\in L^1(\R^m)$ such that
\begin{eqnarray}
|u(x,z,0)|&\le& Ce^{-c|z|-|x|^2/b^2},\label{eq:hard1}\\
|u(x,z,T)|&\le& g(z)e^{-|x|^2/a^2}\label{eq:hard2}
\end{eqnarray}
If $ab<4T$, then $u\equiv0$.
\end{theorem}

\begin{remark}
The condition $s\sigma<4T$ is essentially optimal in the sense that the result does not hold when 
$s\sigma>4T$. Indeed, take $V=0$, $\eps,T>0$
and $u(x,z,0)=p_{\eps T}(x,z)$ so that
$u(x,z,T)=p_{(\eps+i)T}(x,z)$. It follows from Lemma \ref{lem:estheat} that
$$
|u(x,z,0)|\leq C_{\eps} e^{-\frac{\pi}{4T} |z|}e^{-|x|^2\frac{(1-\eps)}{4\eps T}}
$$
while
$$
|u(x,z,T)|\leq C_{\eps,T} e^{-\frac{\pi\eps^2}{4(1+\eps^2)T} |z|}e^{-|x|^2\frac{(1-\eps)\eps}{4(1+\eps^2)T}}
$$
Thus \eqref{eq:hard1} is satisfied with $b^2=\frac{4\eps T}{(1-\eps)}$
 while
\eqref{eq:hard1} is satisfied with $a^2=\frac{4(1+\eps^2)T}{(1-\eps)\eps}$.
It follows that
$$
a^2b^2=16T^2\frac{1+\eps^2}{(1-\eps)^2}
$$
and any number $>16T^2$ can be written in this form.

Note that if $|u(x,z,0)|\leq C e^{-d(x,z)^2/\beta^2}$ and $|u(x,z,T)|\leq C e^{-d(x,z)^2/\alpha^2}$
then, with \eqref{eq:equivdistprecise}, we get that \eqref{eq:hard1}-\eqref{eq:hard2} hold with $a=(1+\eps)\alpha$, $b=(1+\eps)\beta$ and $\eps>0$ arbitrarily small. It follows that, if $\alpha\beta <4T$ then $u\equiv0$. The above example again shows that this is false if $\alpha\beta >4T$.
\end{remark}

\begin{proof}
The proof of this theorem is then done in several steps.

\medskip

\noindent{\bf Step 1.} {\em Reduction to $T=1$.}

\medskip

Note that if $u$ is a solution of \eqref{eq:SchrodHtype} 
on $G\times [0,T]$ and $U(x,z,t)=u(\sqrt{T}x,T z,Tt)$ then $U$ is a solution of
\begin{equation}
\label{eq:scaleshrod}
i\partial_t U=\mathcal{L}_GU+V_T U
\end{equation}
in $G\times [0,1]$ where $V_T(x,z,t)=T V(\sqrt{T}x,Tz,Tt)$.
Note that $V_T$ satisfies Hypothesis \ref{potential} with parameters $\tilde a=a/T,\tilde b=b/T,1$.
Moreover, \eqref{eq:hard1} implies that 
$$
|U(x,z,0)|\leq Ce^{-cT|z|-T|x|^2/b^2}=e^{-cT|z|-|x|^2/\tilde b^2}
$$
while \eqref{eq:hard2} implies that $|U(x,z,1)|\leq g(Tz)e^{-|x|^2/\tilde a^2}$.
Therefore, it is enough to prove Theorem \eqref{th:SchroHtypeUP} for $T=1$.

\medskip

\noindent{\bf Step 2.} {\em Reduction to the Heisenberg group.}

\medskip

This reduction has been pointed out in \cite{BTD,Ra} and is done via a partial Radon transform:

\begin{proposition}\label{prop:HtypetoHeis}
Let $u$ and $V$ be as in Theorem \ref{th:SchroHtypeUP}.
Let $\eta\in\S^{m-1}$ be a unit vector in $\mathbb{R}^m$ and $\eta^\perp=\{\xi\in\mathbb{R}^m:\langle \xi,\eta\rangle=0\}$. Define,
$$
U_\eta(x,s,t)=\int_{\eta^\perp}u(x,s\eta+\tilde{z},t)\d \tilde{z},\ \ x\in\mathbb{R}^{2n},\ s\in\mathbb{R},\ t\in[0,1].
$$

Then $U_\eta$ is a solution of the following Schr\"odinger equation
\begin{equation}
\label{eq:shrodheis}
i\partial_t U_\eta+\mathcal{L}_{\mathbb{H}^n}U_\eta+V(x,t)U_\eta=0,\ \ (x,s,t)\in \mathbb{H}^n\times[0,1],
\end{equation}
where $\mathcal{L}_{\mathbb{H}^n}$ stands for the sub-Laplacian on the Heisenberg group $\mathbb{H}^n$.
Moreover, there exists $C_1$ and $c_1$ positive constants depending only on $C$ and $c$ such that, for $(x,s)\in\mathbb{H}^n$
\begin{eqnarray*}
|U_\eta(x,s,0)|&\le& C_1e^{-c_1|s|-|x|^2/b^2}, \\
|U_\eta(x,s,1)|&\le& \tilde{g}(s)e^{-|x|^2/a^2}
\end{eqnarray*}
where $\dst\tilde{g}(s):=\int_{\eta^\perp}g(s\eta+\tilde{z})\d\tilde{z}$ is the Radon transform of $g$.
\end{proposition}

Note that $\tilde{g}(s)$ is well defined for almost every $s,\eta$ and that $\tilde{g}\in L^1(\R)$ for almost every $\eta$.

\begin{proof}
The fact that $U_\eta$ satisfies \eqref{eq:shrodheis} is established in \cite{BTD,Ra}. It remains to establish the
Gaussian bounds on $U$. Since $\dst
U_\eta(x,s,t)=\int_{\eta^\perp}u(x,s\eta+\tilde{z},t)\d \tilde{z},
$
we have
$$
|U_\eta(x,s,0)|\le \int_{\eta^\perp}|u(x,s\eta+\tilde{z},0)|\d\tilde{z}\le C e^{-|x|^2/b^2}\int_{\eta^\perp}e^{-c|s\eta+\tilde{z}|}\d\tilde{z}.
$$
Taking into account that $\tilde{z}\in\eta^\perp$, and that $\eta$ is a unit vector, we get
$$
|s\eta+\tilde{z}|^2=s^2+|z|^2\Rightarrow |s\eta+\tilde{z}|\ge \frac{|s|+|\tilde{z}|}{\sqrt{2}}.
$$
Using this estimate leads to the desired decay for $U_\eta(x,s,0)$. Moreover,
$$
|U_\eta(x,s,1)|\le \int_{\eta^\perp}|u(x,s\eta+\tilde{z},1)|\d\tilde{z}\le  e^{-|x|^2/a^2}\int_{\eta^\perp}g(s\eta+\tilde{z})\d\tilde{z}
$$
as claimed.
\end{proof}

\medskip

\noindent{\bf Step 3.} {\em Reduction to the magnetic Laplacian on $\R^n$.}

\medskip

We now show that Schr\"odinger equations on the Heisenberg group can be seen as magnetic Schr\"odinger equations.

\begin{proposition}\label{prop:HeistoMag}
Let $u$ and $V$ be as in Theorem \ref{th:SchroHtypeUP}. Let $U_\eta$ be defined in Proposition \ref{prop:HtypetoHeis}.

For $\xi\in\mathbb{R}$, consider the partial Fourier transform of $U$ in the central variable:
$$
f_\xi(x,t)=\int_{\R}U_\eta(x,z,t)e^{i\xi z}\d z.
$$
Then $f_\xi$ is a solution of the following magnetic Schr\"odinger equation
$$
i\partial_t f_\xi+\Delta_Cf_\xi+Vf_\xi=0,
$$
where $\Delta_C=(\nabla-iC_\xi)^2$ and $C_\xi=\frac{M_\xi x}{2}$ with 
$M_\xi=\xi\left(\begin{matrix}0&-I_{n\times n}\\ I_{n\times n}&0\end{matrix}\right)$.

Moreover $f_\xi(x,0)$ admits a holomorphic extension to $f_{\xi+i\eta}(x,0)$ in $\{|\eta|<c\}$. Further,
\begin{equation}
\label{eq:estfxi}
\begin{matrix}
|f_\xi(x,0)|&\le& Ce^{-|x|^2/b^2},\\
|f_\xi(x,1)|&\le& \|\tilde g\|_{L^1}e^{-|x|^2/a^2}.
\end{matrix}
\end{equation}
\end{proposition}

\begin{proof}
We recall that 
$$
\mathcal{L}= \sum_{j=1}^{2n}X_j^2= \Delta_{\mathbb{R}^{2n}}+\frac{1}{4}|x|^2\frac{\partial^2}{\partial z^2}+\sum_{j=1}^n\left(x_j\frac{\partial}{\partial x_{n+j}}-x_{n+j}\frac{\partial}{\partial x_j}\right)\frac{\partial}{\partial z}.
$$

Then,
\begin{eqnarray*}-i\partial_tf_\xi&=&-\int_{\R}i\partial_tUe^{i\xi z}\d z\\&=&\int_{\R}\Delta_{\R^{2n}} Ue^{i\xi z}\d z+\frac14|x|^2\int_{\R}\partial_{zz}Ue^{i\xi z}\d z+\sum_{j=1}^n\int_{\R}\left(x_j\partial_{n+j}\partial_zU-x_{n+j}\partial_j\partial_zU\right)e^{i\xi z}\d z\\&&+V(x,t)\int_{\R}Ue^{i\xi z}\d z
\\&=& \Delta_{\R^{2n}}f_\xi-\frac{\xi^2}{4}|x|^2f_\xi-i\xi\sum_{j=1}^n\left(x_j\partial_{n+j}f_\xi-x_{n+j}\partial_jf_\xi\right)+Vf_\xi=\Delta_Cf_\xi+Vf_\xi
\end{eqnarray*}
as claimed.

The last part of the proposition is immediate using the definition of $f_\xi$ as a partial Fourier transform in the central variable.
\end{proof}

\medskip

\noindent{\bf Step 4.} {\em Conclusion}

Using the previous step, $f\xi$ is a solution of
$$
i\partial_t f_\xi+\Delta_Cf_\xi+Vf_\xi=0,
$$
that satisfies the Gaussian estimates \eqref{eq:estfxi}. Further, since $ab<4$, there exists $\delta>0$ such that $ab<4\frac{\sin\delta}{\delta}$. Now, using Theorem 1.11 in \cite{CF} (in the case $A\equiv0$, which implies that the result is valid in any even dimension) combined with the previous results that give information about the evolution and decay of $f_\xi$, we conclude that,
for $0<\xi<\delta$, $t\in[0,1]$,  $f_\xi(x,t)=0$.
In particular, for $t=0$ we get that $f_\xi(x,0)=0$ for $0<\xi<\delta$. But, from Proposition \ref{prop:HeistoMag},
we know that $f_\xi(x,0)$ admits an holomorphic extension to a strip $\{\xi+i\zeta : |\zeta|<c\}$. It follows that
$f_\xi(x,0)=0$ on the whole real line $(\zeta=0)$. 

As $\xi$ is arbitrary and
$f_\xi(x,0)$ is the partial Fourier transform of $U_\eta(x,s,0)$, we conclude that $U_\eta(x,s,0)=0$.

Finally, $U_\eta$ is the Radon transform of $u$, we conclude that $u=0$, which is the desired result..
\end{proof}

As an immediate corollary, we can now obtain the following result, which is stated in a slightly more restrictive form as Theorem \ref{th:intro1} in the introduction:

\begin{corollary}
\label{cor:SchroHtypeUP}
Let $T>0$ and $s,\sigma>0$. Let $V$ be a potential satisfying Hypothesis \ref{potential} with parameters 
$2\sqrt{\sigma},2\sqrt{s},T$.
Let $u\in C^1([0,T], H^1(G))$ be a solution of \eqref{eq:SchrodHtype}. Assume that there is a $C>0$ such that
\begin{eqnarray*}
|u(x,z,0)|&\le& C p_{s}(x,z).\\
|u(x,z,T)|&\le& C p_{\sigma}(x,z).
\end{eqnarray*}
If $s\sigma<T^2$, then $u\equiv0$.
\end{corollary}

\begin{remark} The case $V=0$ of this theorem has been obtained by 
Ben Sa\"id, Thangavellu, Doga \cite{BTD} for $H$-type groups and by
Ludwig and M\"uller \cite{LM} for general step 2 nilpotent Lie groups.
\end{remark}

\begin{proof}[Proof of Corollary \ref{cor:SchroHtypeUP}] 
From Lemma \ref{lem:estheat} we get
\begin{eqnarray*}
|u(x,z,0)|&\le& C_\eps \exp\left(-\frac{1}{4s}\bigl((1-\eps)|x|^2+\pi\eps|z|\bigr)\right)\\
|u(x,z,T)|&\le& C_\eps \exp\left(-\frac{1}{4\sigma}\bigl((1-\eps)|x|^2+\pi\eps|z|\bigr)\right).
\end{eqnarray*}
Let $a_\eps=2\sqrt{\frac{\sigma}{1-\eps}}$ and $b_\eps=2\sqrt{\frac{s}{1-\eps}}$. As $s\sigma<T^2$,
for $\eps$ small enough, $a_\eps b_\eps=\frac{4}{1-\eps}\sqrt{s\sigma}<4T$.

Further $V$ satisfies Hypothesis \ref{potential} with parameters $a_\eps,b_\eps,T$. Applying Theorem \ref{th:SchroHtypeUP} we thus get $u=0$.
\end{proof}

\section*{Acknowledgments}

This study has been carried out with financial support from the French State, managed
by the French National Research Agency (ANR) in the frame of the ``Investments for
the future'' Programme IdEx Bordeaux - CPU (ANR-10-IDEX-03-02).

A. F.-B. acknowledges financial support from ERCEA Advanced Grant 2014 669689 - HADE, the MINECO project MTM2014-53145-P and the Basque Government project IT-641-13.

P.J. acknowledges financial support from the French ANR program ANR-12-BS01-0001 (Aventures)
from the Austrian-French AMADEUS project 35598VB - Ch
argeDisq and from the Tunisian-French CMCU/Utique project 15G1504.

S.P.E. acknowledges financial support from the Mexican Grant PAPIIT-UNAM IN106418.

The second author thanks L. Bouattour for valuable conversation at the early stage of this project.

\end{document}